\documentclass[12pt, reqno]{amsart}

\usepackage{
amsmath,
amssymb,
amsthm,
mathrsfs,
amsfonts,
ytableau,
enumitem,
comment,
upgreek,
soul,
yhmath,
mathtools,
shuffle,
kotex,
array,
caption,
tabularx,
}

\usepackage{thmtools}

\usepackage[linktocpage]{hyperref}

\hypersetup{
    colorlinks=true,
    linkcolor = blue,
    citecolor=magenta,
    urlcolor=cyan,
}

\usepackage[capitalise, nameinlink]{cleveref}

\crefname{equation}{}{}
\crefname{figure}{{\sc Figure}}{{\sc Figure}}

\numberwithin{equation}{section} 
\numberwithin{figure}{section}
\numberwithin{table}{section}

\newtheorem{theorem}{Theorem}[section]

\newtheorem{lemma}[theorem]{Lemma}

\newtheorem{conjecture}[theorem]{Conjecture}

\newtheorem*{claim*}{Claim}

\theoremstyle{definition}
\newtheorem{algorithm}[theorem]{Algorithm}
\newtheorem{example}[theorem]{Example}

\usepackage[euler]{textgreek}

\usepackage{tikz,tikz-cd}

\usepackage[colorinlistoftodos, textwidth = 2.3cm]{todonotes}

\ytableausetup{mathmode, boxsize=1.2em}

\def\Z{\mathbb Z}
\def\C{\mathbb C}

\newcommand{\nc}{\newcommand}

\nc{\SG}{\mathfrak{S}}
\nc{\SRT}{\mathrm{SRT}}
\nc{\comp}{\mathrm{comp}}
\nc{\Des}{\mathrm{Des}}
\nc{\set}{\mathrm{set}}
\nc{\ch}{\mathrm{ch}}
\nc{\id}{\mathrm{id}}
\nc{\Sym}{\mathrm{Sym}}
\nc{\QSym}{\mathrm{QSym}}
\nc{\Ext}{\mathrm{Ext}}
\nc{\source}{\mathsf{source}}
\nc{\sink}{\mathsf{sink}}
\nc{\len}{\mathsf{len}}
\nc{\ind}{\mathtt{index}}
\nc{\col}{\mathsf{col}}
\nc{\row}{\mathsf{row}}
\nc{\rad}{\mathrm{rad}}
\nc{\IGLT}{\mathrm{IGLT}}
\nc{\SYT}{\mathrm{SYT}}
\nc{\Top}{\mathsf{Top}}
\nc{\Bot}{\mathsf{Bot}}
\nc{\pr}{\mathsf{pr}}
\nc{\Gr}{\mathrm{Gr}}
\nc{\RS}{\star}

\nc{\calS}{\mathcal{S}}
\nc{\calI}{\mathcal{I}}
\nc{\calR}{\mathcal{R}}
\nc{\calG}{\mathcal{G}}
\nc{\calP}{\mathcal{P}}
\nc{\calE}{\mathcal{E}}

\nc{\bfS}{\mathbf{S}}
\nc{\bfF}{\mathbf{F}}
\nc{\bfP}{\mathbf{P}}
\nc{\bfG}{\mathbf{G}}

\nc{\rmr}{\mathrm{r}}
\nc{\rmc}{\mathrm{c}}
\nc{\rmt}{\mathrm{t}}
\nc{\rmw}{\mathrm{w}}
\nc{\rmv}{\mathrm{v}}

\nc{\sfB}{\mathsf{B}}
\nc{\sfb}{\mathsf{b}}
\nc{\sft}{\mathsf{t}}
\nc{\sfp}{\mathsf{p}}
\nc{\sfq}{\mathsf{q}}
\nc{\sfA}{\mathsf{A}}
\nc{\sfD}{\mathsf{D}}

\nc{\scrT}{\mathscr{T}}

\nc{\tH}{\mathtt{H}}
\nc{\tC}{\mathtt{C}}
\nc{\tyd}{\mathtt{yd}}
\nc{\trd}{\mathtt{rd}}
\nc{\tHL}{\mathtt{HL}}

\nc{\sfhA}{\widehat{\mathsf{A}}}
\nc{\sfread}{\mathsf{read}}
\nc{\sfLread}{\Phi}
\nc{\sfhD}{\widehat{\mathsf{D}}}

\nc{\itread}{\mathit{read}}

\nc{\ra}{\rightarrow}
\nc{\opi}{\overline{\pi}}
\nc{\bal}{{\boldsymbol{\upalpha}}}
\nc{\bfpi}{\boldsymbol{\uppi}}
\nc{\tGam}{\widetilde{\Gamma}}
\nc{\hGam}{\widehat{\Gamma}}

\nc{\GS}[1]{U_{#1}}
\nc{\GSm}[2]{U_{#1;#2}}
\nc{\grco}[2]{(\underline{#1,#2})}
\nc{\IGLTm}[2]{\mathrm{IGLT}(#1)_{#2}}
\nc{\Par}{\mathsf{Par}}
\nc{\sh}{\mathrm{sh}}
\nc{\oPsi}{\overline{\Psi}}

\nc{\Hmod}{H_n(0)\text{-mod}}

\nc{\yh}[1]{\todo[size=\tiny,color=cyan!10]{#1 \\ \hfill --- Young-Hun}}
\nc{\YH}[1]{\todo[size=\tiny,inline,color=cyan!10]{#1
		\\ \hfill --- Young-Hun}}

\nc{\nt}[1]{\todo[size=\tiny,color=exgreen!10]{#1 \\ \hfill --- Note}}
\nc{\NT}[1]{\todo[size=\tiny,inline,color=exgreen!10]{#1
		\\ \hfill --- Note}}

\definecolor{purple}{rgb}{0.44, 0.0, 1.0}
\definecolor{exgreen}{cmyk}{0.8, 0.1, 1, 0}

\definecolor{yhblue}{rgb}{0,0,0.6}

\newenvironment{red}{\relax\color{red}}{\hspace*{.5ex}\relax}
\newenvironment{blue}{\relax\color{yhblue}}{\hspace*{.5ex}\relax}
\newenvironment{green}{\relax\color{wsgreen}}{\hspace*{.5ex}\relax}
\newenvironment{magenta}{\relax\color{magenta}}{\hspace*{.5ex}\relax}

\nc{\ber}{\begin{red}}
\nc{\er}{\end{red}}
\nc{\beb}{\begin{blue}}
\nc{\eb}{\end{blue}}
\nc{\bema}{\begin{magenta}}
\nc{\ema}{\end{magenta}}
\nc{\begr}{\begin{green}}
\nc{\egr}{\end{green}}

\title[A representation-theoretic interpretation]{A representation-theoretic interpretation of the Schur expansion of two-row genomic Schur functions}

\author[Y.-H. Kim]{Young-Hun Kim}
\address{Department of Mathematics, Seoul Women’s University, Seoul 01797, Republic of Korea}
\email{yhkim@swu.ac.kr, ykim.math@gmail.com}

\thanks{
The author was supported by Basic Science Research Program through the National Research Foundation of Korea(NRF) funded by the Ministry of Education (No. RS-2023-00240377) and a research grant from Seoul Women’s University(2025-0241).
}



\begin{document}

\maketitle

\begin{abstract}
Genomic Schur functions were introduced by Pechenik and Yong in connection with the $K$-theory of Grassmannians. Pechenik proved that genomic Schur functions admit a positive expansion in the basis of fundamental quasisymmetric functions and, for partitions with two parts, a positive expansion in the Schur basis. Later, Kim and Yoo constructed $0$-Hecke modules associated with genomic Schur functions and conjectured that the latter expansion admits a representation-theoretic interpretation in terms of $0$-Hecke modules. In this paper, we prove the conjecture of Kim and Yoo, thereby obtaining a representation-theoretic interpretation of the Schur expansion in the two-row case.
\end{abstract}

\section{Introduction}
 
Genomic tableaux were introduced by Pechenik and Yong \cite{17PY} in connection with the $K$-theory of Grassmannians. 
In \cite{17PY2}, Pechenik and Yong defined a symmetric function $U_\lambda$, called the genomic Schur function, as a generating function for genomic tableaux of shape $\lambda$ for each partition $\lambda$. 
They showed that the family $\{U_\lambda\}$ forms a basis of the ring of symmetric functions, and observed that genomic Schur functions are not Schur-positive in general \cite[Example~6.7]{17PY2}. 
This makes it natural to seek positive expansions of $U_\lambda$ in other bases and to ask whether such expansions admit a further interpretation.

A first answer to this question was given by Pechenik \cite{20Pechenik}, who proved that genomic Schur functions are fundamental-positive.
Let $\IGLT(\lambda)_m$ be the set of increasing gapless tableaux of shape $\lambda$ with maximum entry $m$.
Then
\begin{align}\label{eq: expansion of U in F}
U_\lambda
=
\sum_{1 \le m \le |\lambda|}
\ \sum_{T \in \IGLT(\lambda)_m} F_{\comp(T)},
\end{align}
where $F_{\comp(T)}$ denotes the fundamental quasisymmetric function indexed by the descent composition of $T$. 
This formula expresses each homogeneous component of $U_\lambda$ in terms of increasing gapless tableaux and naturally leads to the question of whether this expansion admits a representation-theoretic interpretation.

The relevant representation-theoretic framework is provided by the quasisymmetric characteristic introduced by Duchamp, Krob, Leclerc, and Thibon in \cite{96DKLT}. 
They constructed a ring isomorphism
\[
\ch:\bigoplus_{n\ge 0} G_0(H_n(0)\text{-mod}) \longrightarrow \QSym,
\]
where $G_0(H_n(0)\text{-mod})$ denotes the Grothendieck group of the category of finite-dimensional $H_n(0)$-modules and $\QSym$ denotes the ring of quasisymmetric functions.
In view of this correspondence, various quasisymmetric and symmetric functions have been studied via suitable $0$-Hecke modules; see, for instance, \cite{22BS,15BBSSZ,20CKNO2,24KY,19Searles,15TW}. 
In particular, Searles \cite{19Searles} constructed, for each partition $\mu \vdash m$, an indecomposable $H_m(0)$-module $X_\mu$ satisfying $\ch([X_\mu]) = s_\mu$, where $s_\mu$ denotes the Schur function indexed by $\mu$.

Motivated by Pechenik's fundamental-positive formula, Kim and Yoo \cite{24KY} constructed, for each partition $\lambda$ and each integer $m$ with $1 \le m \le |\lambda|$, an $H_m(0)$-module $G_{\lambda;m}$ whose quasisymmetric characteristic is the $m$th homogeneous component $U_{\lambda;m}$ of $U_\lambda$. 
They also introduced an equivalence relation $\sim_{\lambda;m}$ on $\IGLT(\lambda)_m$ and obtained a direct sum decomposition
\begin{align}\label{intro: summand of GE}
G_{\lambda;m}=\bigoplus_{E\in\mathcal E_{\lambda;m}} G_E.
\end{align}
Here, $\mathcal E_{\lambda;m}$ denotes the set of equivalence classes of $\IGLT(\lambda)_m$ under $\sim_{\lambda;m}$, and $G_E$ denotes the $H_m(0)$-submodule of $G_{\lambda;m}$ spanned by $E$ for each $E \in \mathcal E_{\lambda;m}$.

By contrast, Schur-positivity of $U_\lambda$ holds only in special cases.
In \cite[Proposition~4.3]{20Pechenik}, Pechenik constructed a bijection from $\IGLT(\lambda)_m$ to $\bigcup_{\mu \in \Par(\lambda;m)}\SYT(\mu)$ and used it to prove that if $\lambda \vdash n$ has two parts, then
\begin{align}\label{intro: Pechenik expansion}
U_\lambda
=
\sum_{\max\{\lambda_1,\lambda_2+1\} \le m \le n}
\ \sum_{\mu\in\Par(\lambda;m)} s_\mu,
\end{align}
where $\Par(\lambda;m)$ is defined in \cref{eq: def of Par}.
In \cite[Conjecture~7.1]{24KY}, Kim and Yoo proposed a representation-theoretic interpretation of this formula.
Roughly speaking, they asked whether the summands $G_E$ can be partitioned according to the Schur summands $s_\mu$, and whether each $G_E$ can be realized as a successive quotient in a filtration of $X_\mu$.

The purpose of this paper is to prove the conjecture of Kim and Yoo \cite[Conjecture~7.1]{24KY}. 
For each partition $\lambda$ with $\ell(\lambda) = 2$ and each integer $m$ satisfying $\max\{\lambda_1,\lambda_2+1\}\le m\le n$, we partition $\calE_{\lambda;m}$ into subsets $\calE_\mu$ according to the Schur summands $s_\mu$ in Pechenik's expansion \cref{intro: Pechenik expansion}. 
Moreover, we prove in \cref{thm: main} that for each $\mu \in \Par(\lambda;m)$, the successive quotients in a filtration of the $H_m(0)$-module $X_\mu$ are, up to isomorphism, exactly the modules $G_E$ with $E \in \calE_\mu$.
In this way, when $\ell(\lambda) = 2$, the Schur expansion of $U_\lambda$ admits a representation-theoretic interpretation in terms of $0$-Hecke modules.
The proof uses Pechenik's bijection together with a suitable partial order on $\calE_{\lambda;m}$ to construct filtrations of the corresponding modules $X_\mu$.

\section{Preliminaries}\label{Sec: prem}

For integers $m$ and $n$, define
$[m,n]$ to be the set $\{k \in \Z \mid m \le k \le n\}$ if $m \le n$, and $\emptyset$ otherwise.
Throughout this paper, we assume that $n$ is a nonnegative integer.

\subsection{Compositions, partitions, and Young diagrams}\label{subsec: comp and diag}

A \emph{composition} $\alpha$ of $n$, denoted by $\alpha \models n$, is a finite ordered list of positive integers $(\alpha_1, \alpha_2, \ldots, \alpha_k)$ satisfying $\sum_{i=1}^k \alpha_i = n$.
We call $k =: \ell(\alpha)$ the \emph{length} of $\alpha$ and $n =:|\alpha|$ the \emph{size} of $\alpha$. For convenience we define the empty composition $\emptyset$ to be the unique composition of size and length $0$.

Given $\alpha = (\alpha_1, \alpha_2, \ldots,\alpha_{\ell(\alpha)}) \models n$ and $I = \{i_1 < i_2 < \cdots < i_l\} \subset [1, n-1]$, 
let $\set(\alpha) := \{\alpha_1, \alpha_1+\alpha_2, \ldots, \alpha_1 + \alpha_2 + \cdots + \alpha_{\ell(\alpha)-1}\}$ and $\comp(I) := (i_1, i_2 - i_1, \ldots, n-i_l)$.
The set of compositions of $n$ is in bijection with the set of subsets of $[1, n-1]$ under the correspondence $\alpha \mapsto \set(\alpha)$ (or $I \mapsto \comp(I)$).

For $\lambda = (\lambda_1,\lambda_2,\ldots, \lambda_{\ell(\lambda)}) \models n$ if it satisfies the inequalities $\lambda_1 \ge \lambda_2 \ge \cdots \ge \lambda_{\ell(\lambda)}$, then we say that $\lambda$ is a \emph{partition} of $n$ and denote it by $\lambda \vdash n$.
For $\lambda = (\lambda_1, \lambda_2, \ldots, \lambda_{\ell(\lambda)}) \vdash n$, 
we define the \emph{Young diagram} $\tyd(\lambda)$ of $\lambda$ by a left-justified array of $n$ boxes where the $i$th row from the top has $\lambda_i$ boxes for $1 \le i \le \ell(\lambda)$.
We write $(i,j)$ for the box
in the $i$th row and $j$th column.
Then $(i,j) \in \tyd(\lambda)$ means $1 \le i \le \ell(\lambda)$ and $1 \le j \le \lambda_i$.
A \emph{filling} of $\tyd(\lambda)$ is a function $T: \tyd(\lambda) \ra \Z_{>0}$. 
For any filling $T$ of $\tyd(\lambda)$, we define $\max(T)$ to be the largest integer among $T(i,j)$ with $(i,j) \in \tyd(\lambda)$.
Throughout this paper, we extend $T$ to $\Z^2$ by setting
$T(i,j)=\infty$ for $(i,j) \in \Z_{>0} \times \Z_{>0} \setminus \tyd(\lambda)$
and $T(i,j) = -\infty$ for $(i,j) \notin \Z_{>0} \times \Z_{>0}$.

\subsection{The $0$-Hecke algebra and the quasisymmetric characteristic}\label{subsec: 0-Hecke alg}
The symmetric group $\SG_n$ is generated by simple transpositions $s_i$ with $1 \le i \le n-1$.
An expression for $\sigma \in \SG_n$ of the form $s_{i_1} s_{i_2} \cdots s_{i_p}$ that uses the minimal number of simple transpositions is called a \emph{reduced expression} for $\sigma$. 
The number of simple transpositions in any reduced expression for $\sigma$, denoted by $\ell(\sigma)$, is called the \emph{length} of $\sigma$.

The $0$-Hecke algebra $H_n(0)$ is the unital $\C$-algebra generated by $\pi_1, \pi_2, \ldots,\pi_{n-1}$ subject to the following relations:
\begin{align*}
\pi_i^2 &= \pi_i \quad \text{for $1\le i \le n-1$},\\
\pi_i \pi_{i+1} \pi_i &= \pi_{i+1} \pi_i \pi_{i+1}  \quad \text{for $1\le i \le n-2$},\\
\pi_i \pi_j &=\pi_j \pi_i \quad \text{if $|i-j| \ge 2$}.
\end{align*}
According to \cite{79Norton}, there are exactly $2^{n-1}$ distinct irreducible
$H_n(0)$-modules which are naturally parametrized by the compositions of $n$.
For $\alpha \models n$, the irreducible module $\bfF_\alpha$ corresponding to $\alpha$
is the $1$-dimensional $H_n(0)$-module spanned by a vector $v_\alpha$ with $H_n(0)$-action given by
$\pi_i \cdot v_\alpha = 0$ if $i \in \set(\alpha)$ and
$\pi_i \cdot v_\alpha = v_\alpha$ if $i \notin \set(\alpha)$ for $1 \le i \le n-1$.
Let $G_0(H_n(0))$ be the Grothendieck group of the category of finite-dimensional $H_n(0)$-modules.
For an $H_m(0)$-module $M$ and an $H_n(0)$-module $N$, define the induction product of $M$ and $N$ to be $M \otimes N \uparrow_{H_m(0) \otimes H_n(0)}^{H_{m+n}(0)}$.
Then $\bigoplus_{n \ge 0} G_0(H_n(0))$ has a ring structure under the induction product.

A quasisymmetric function is a power series of bounded degree in variables $x_{1},x_{2},x_{3},\ldots$ with integer coefficients such that for any composition $\alpha=(\alpha_1,\alpha_2,\ldots,\alpha_k)$, the coefficients of $x_{i_1}^{\alpha_1} x_{i_2}^{\alpha_2}\cdots x_{i_k}^{\alpha_k}$ are equal for all $i_1<i_2<\cdots<i_k$.
Let $\QSym$ be the ring of quasisymmetric functions.
For $\alpha \models n$, the \emph{fundamental quasisymmetric function} $F_\alpha$ is defined by $F_\emptyset = 1$ and
\begin{align*}
F_\alpha = \sum_{\substack{1 \le i_1 \le i_2 \le \cdots \le i_n \\ i_j < i_{j+1} \text{ if } j \in \set(\alpha)}} x_{i_1} x_{i_2} \cdots x_{i_n}.
\end{align*}
It is shown in \cite{83Gessel} that $\{F_\alpha \mid \text{$\alpha$ is a composition}\}$ is a basis for $\QSym$.

In~\cite{96DKLT}, Duchamp, Krob, Leclerc, and Thibon introduced the
\emph{quasisymmetric characteristic}, a ring isomorphism
\[
\ch : \bigoplus_{n \ge 0} G_0(H_n(0)) \to \QSym, \quad [\bfF_{\alpha}] \mapsto F_{\alpha}.
\]

\subsection{0-Hecke modules from standard Young tableaux and from increasing gapless tableaux}

In this subsection, $\lambda$ denotes a partition of $n$.

\subsubsection{$H_n(0)$-modules from standard Young tableaux}

A \emph{standard Young tableau} (SYT) of shape $\lambda$ is a filling $S$ of $\tyd(\lambda)$ with $[1,n]$ such that each number occurs exactly once, the entries increase from left to right along rows, and from top to bottom along columns.
We denote by $\SYT(\lambda)$ the set of all standard Young tableaux of shape $\lambda$.
For $S \in \SYT(\lambda)$ and $i \in [1,n-1]$, we say that $i$ is a \emph{descent of $S$} if $i$ appears strictly above $i+1$ in $S$.
We denote by $\Des(S)$ the set of all descents of $S$. 
It is well known that $s_\lambda = \sum_{S \in \SYT(\lambda)} F_{\comp(\Des(S))}$.

In \cite{19Searles}, Searles defines the $H_n(0)$-module, denoted by $X_\lambda$, whose underlying space is the $\C$-span of $\SYT(\lambda)$ and whose action is defined as follows: 
for each $1 \le i \le n-1$ and $S \in \SYT(\lambda)$,
\begin{align*}
\pi_i \cdot S = \begin{cases}
S & \text{if $i$ is strictly left of $i+1$ in $S$},\\
0 & \text{if $i$ and $i+1$ are in the same column of $S$},\\
s_i \cdot S & \text{if $i$ is strictly right of $i+1$ in $S$}.
\end{cases}
\end{align*}
Here, $s_i \cdot S$ is the tableau obtained from $S$ by swapping $i$ and $i+1$. 
The following properties of $X_\lambda$ were shown in \cite{19Searles}:
\begin{enumerate}[label= {\rm(\roman*)}]
\item
$\ch([X_\lambda]) = s_\lambda$, where $s_\lambda$ denotes the Schur function indexed by $\lambda$.
\item
$X_\lambda$ is indecomposable.
\end{enumerate}

\subsubsection{$H_n(0)$-modules from increasing gapless tableaux}
\label{subsec: Hn from IGLT}

An \emph{increasing gapless tableau} of shape $\lambda$ is a filling of $\tyd(\lambda)$ satisfying the following conditions:
\begin{enumerate}[label = {\rm (\arabic*)}]
\item the entries in each row strictly increase from left to right,
\item the entries in each column strictly increase from top to bottom, and
\item the set $T^{-1}(k)$ is nonempty for all $1 \le k \le \max(T)$.
\end{enumerate}
Let $\IGLT(\lambda)$ be the set of all increasing gapless tableaux of shape $\lambda$.
Given $1 \le m \le n$, let $\IGLTm{\lambda}{m} := \{T \in \IGLT(\lambda) \mid \max(T) = m\}$.

In~\cite{20Pechenik}, Pechenik gave the following alternative description of the genomic Schur function, originally defined by Pechenik and Yong in \cite{17PY2}, using increasing gapless tableaux:
\[
\GS{\lambda} := \sum_{1 \le m \le n} \left( \sum_{T \in \IGLT(\lambda)_{m}} F_{\comp(T)} \right)
\]
In this paper, we adopt this formula as the definition of the genomic Schur function.
For each $1 \le m \le n$, let $\GSm{\lambda}{m}$ be the $m$th degree homogeneous component of $\GS{\lambda}$.

Let $1 \le m \le n$ and $T \in \IGLT(\lambda)_m$.
For $1 \le i \le m$, let
\[
(r_\sft^{(i)}(T), c_\sft^{(i)}(T)) \quad \text{and} \quad
(r_\sfb^{(i)}(T), c_\sfb^{(i)}(T))
\]
denote the coordinates of the topmost and bottommost boxes in $T^{-1}(i)$, respectively.
If $T\in \SYT(\lambda)$ or $|T^{-1}(i)|=1$, we set
\[
(r^{(i)}(T),c^{(i)}(T))
:=(r_\sfb^{(i)}(T),c_\sfb^{(i)}(T))
=(r_\sft^{(i)}(T),c_\sft^{(i)}(T)).
\]
When $T$ is clear from the context, we drop the argument $T$ and write simply
$r_\sfb^{(i)}, c_\sfb^{(i)}, r_\sft^{(i)}$, and $c_\sft^{(i)}$.

For $1 \le i \le m-1$, we call $i$ a \emph{descent of $T$} if $r_\sft^{(i)} < r_\sfb^{(i+1)}$ and we denote by $\Des(T)$ the set of all descents of $T$.
We say that a descent $i$ of $T$ is \emph{attacking} if either of the following holds:
\begin{enumerate}[label = \textrm{(\alph*)}]
\item
there exists $(j,k) \in \tyd(\lambda)$ such that $T(j,k)  = i$ and $T(j +1, k) = i + 1$, or

\item 
there exists a box $B \in T^{-1}(i+1)$ that lies weakly above $(r_\sfb^{(i)}, c_\sfb^{(i)})$.
\end{enumerate}
A descent that is not attacking is called \emph{non-attacking}.
We note that if $i$ is a non-attacking descent of $T$, then all $(i+1)$'s lie strictly below and strictly left of $(r_\sfb^{(i)}, c_\sfb^{(i)})$.

In \cite{24KY}, Kim and Yoo define the $H_n(0)$-module, denoted by $\bfG_{\lambda;m}$, whose underlying space is the $\C$-span of $\IGLTm{\lambda}{m}$ and whose action is defined as follows:
for each $1 \le i \le m-1$ and $T \in \IGLTm{\lambda}{m}$,
\begin{align*}
\pi_i \cdot T := 
\begin{cases}
T & \text{if $i$ is not a descent of $T$,}\\
0 & \text{if $i$ is an attacking descent of $T$,}\\ 
s_i \cdot T & \text{if $i$ is a non-attacking descent of $T$}.
\end{cases}
\end{align*}
Here, $s_i \cdot T$ is the tableau obtained from $T$ by swapping all $i$'s and $i+1$'s.
It was shown in \cite[Proposition 5.6]{24KY} that $\ch([\bfG_{\lambda;m}]) = \GSm{\lambda}{m}$.

In contrast to $X_\lambda$, the module $\bfG_{\lambda;m}$ is not indecomposable.
Below we briefly recall the decomposition of $\bfG_{\lambda;m}$.

We denote by $\grco{i}{j}$ the lattice point on the $(i+1)$st horizontal line from the top and the $(j+1)$st vertical line from the left.
For each $T \in \IGLTm{\lambda}{m}$, let 
\[
\calI(T) := \left\{ i \in [1,m] \; \middle| \; |T^{-1}(i)| > 1 \right\}.
\]
For $i \in \calI(T)$, let $\Gamma_i(T)$ be the lattice path from $\grco{r_\sfb^{(i)}}{c_\sfb^{(i)} - 1}$ to $\grco{r_\sft^{(i)} - 1}{c_\sft^{(i)}}$ satisfying the following two conditions:

\begin{enumerate}[label = {\rm (\roman*)}, leftmargin = 4ex, itemsep = 1ex]
\item 
if $\Gamma_i(T)$ passes between two adjacent boxes horizontally, then the entry in the above box is strictly smaller than $i$ and the entry in the below box is weakly greater than $i$;

\item 
if $\Gamma_i(T)$ passes between two adjacent boxes vertically, then the entry in the left box is strictly smaller than $i$ and the entry in the right box is weakly greater than $i$.
\end{enumerate}
The equivalence relation $\sim_{\lambda;m}$ on $\IGLTm{\lambda}{m}$ is defined by $T_{1} \sim_{\lambda;m} T_{2}$ if and only if 
\[
\left\{ \left(\Gamma_i(T_1), T_1^{-1}(i) \right) \; \middle| \; i \in \calI(T_1) \right\} 
= \left\{ \left(\Gamma_i(T_2), T_2^{-1}(i) \right) \; \middle| \; i \in \calI(T_2) \right\}.
\]
We denote by $\calE_{\lambda;m}$ the set of equivalence classes of $\IGLTm{\lambda}{m}$ with respect to $\sim_{\lambda;m}$.
In \cite[Theorem 3.11]{24KY}, it was shown that for each $E \in \calE_{\lambda;m}$ $\C E$ is closed under the $H_m(0)$-action. 
Consequently
\[
\bfG_{\lambda;m} = \bigoplus_{E \in \calE_{\lambda;m}} \bfG_E.
\]

\subsubsection{Representation-theoretic conjecture for Schur expansions of $U_\lambda$ for $\lambda$ with $\ell(\lambda)=2$}

Let $\lambda = (\lambda_1, \lambda_2) \vdash n$, and set $l_\lambda = \max \{\lambda_1, \lambda_2 +1\}$.
For $m \in \Z$ with $l_\lambda \le m < n$, define
\begin{align}\label{eq: lambda(i)}
\begin{aligned}    
\lambda_m^{(1)} &:= 
\begin{cases}
(\lambda_1 - k_m, \lambda_2 - k_m, 1^{k_m})
& \text{if $\lambda_2 > k_m + 1$,}
\\
\emptyset & \text{if $\lambda_2 \le k_m + 1$,}
\end{cases}
\\[1ex]
\lambda_m^{(2)} &:= 
\begin{cases}
(\lambda_1 - k_m, \lambda_2 - k_m + 1, 1^{k_m-1})
& \text{if $\lambda_2 > k_m$,}
\\
\emptyset & \text{if $\lambda_2 \le k_m$,}
\end{cases}
\end{aligned}
\end{align}
where $k_m := n-m$, and $1^{k_m}$ is the partition $(1,1,\ldots,1)$ of length $k_m$.
Note that for $x=1,2$, whenever $\lambda^{(x)}_m \neq \emptyset$, it is a partition of $m$.
For $m \in \Z$ with $l_\lambda \le m \le n$, define 
\begin{align}\label{eq: def of Par}
\begin{aligned}
\mathsf{Par}(\lambda; m) & := 
\begin{cases}
\left\{
\lambda_m^{(1)}
\right\}
\setminus \{\emptyset\} 
& \text{if $m=n$ or $\lambda_1 = \lambda_2$,} \\[1ex]
\left\{
\lambda_m^{(1)}, \lambda_m^{(2)}
\right\} \setminus \{\emptyset\} 
& \text{if $m < n$ and $\lambda_1 > \lambda_2$.}
\end{cases}
\end{aligned}
\end{align}
In \cite[Proposition 4.3]{20Pechenik}, Pechenik proved that
\begin{align}\label{eq: schur expansion for genomic}
U_\lambda = \sum_{l_\lambda \le m \le n} \  \sum_{\mu \in \mathsf{Par}(\lambda; m)} s_{\mu},
\end{align}
with the convention that $s_\mu := 0$ if $\mu$ is not a partition.

In \cite{24KY}, Kim and Yoo proposed the following conjecture providing a representation-theoretic interpretation of \cref{eq: schur expansion for genomic}.
\begin{conjecture}{\rm (\cite[Conjecture 7.1]{24KY})}
\label{conj: Kim and Yoo}
Let $\lambda$ be a partition with $\ell(\lambda) = 2$.
\footnote{In \cite{24KY}, the conjecture is stated for partitions $\lambda$ with $\ell(\lambda)\le 2$. When $\ell(\lambda)=1$, we have $\IGLT(\lambda)=\SYT(\lambda)$, so the conjecture holds trivially. Thus, in this paper, we restrict our attention to partitions $\lambda$ with $\ell(\lambda)=2$.}
For each $l_\lambda \le m \le n$, there exists a partition $\{\calE_\mu \mid \mu \in \mathsf{Par}(\lambda; m)\}$ of $\calE_{\lambda;m}$ satisfying the following two conditions:
\begin{enumerate}[label = {\bf C\arabic*.}]
\item For each $\mu \in \mathsf{Par}(\lambda; m)$, $\sum_{E \in \calE_\mu} \ch([\bfG_E]) = s_{\mu}$.
\item For each $\mu \in \mathsf{Par}(\lambda; m)$, there exists a total order $\prec_\mu$ on $\calE_\mu = \{E_1 \prec_\mu E_2 \prec_\mu \cdots \prec_\mu E_{|\calE_\mu|}\}$ and a filtration 
\[
M^\mu_0 = \{0\} \subseteq M^\mu_1 \subseteq M^\mu_2 \subseteq \cdots \subseteq M^\mu_{|\calE_\mu|} = X_\mu
\]
of $H_m(0)$-modules such that $\bfG_{E_i} \cong M^\mu_{i}/M^\mu_{i-1}$ for all $1 \le i \le |\calE_\mu|$.
\end{enumerate}
\end{conjecture}

Indeed, if this conjecture holds, then for each $m$ with $l_\lambda \le m \le n$, we have
\begin{align*}
U_{\lambda;m}
&= \ch([\bfG_{\lambda;m}])
= \sum_{E \in \calE_{\lambda;m}} \ch([\bfG_E])
\overset{(\textbf{C1})}{=} \sum_{\mu \in \Par(\lambda;m)} \sum_{E \in \calE_{\mu}} \ch([\bfG_E]) \\
&\overset{(\textbf{C2})}{=} \sum_{\mu \in \Par(\lambda;m)} \sum_{1 \le i \le |\calE_\mu|} \ch([M^\mu_i / M^\mu_{i-1}])
= \sum_{\mu \in \Par(\lambda;m)} \ch([X_\mu])
= \sum_{\mu \in \Par(\lambda;m)} s_\mu,
\end{align*}
which agrees with \cref{eq: schur expansion for genomic}.
This shows that the conjecture provides a representation-theoretic interpretation of the Schur expansion.

\section{Representation-theoretic interpretation of the Schur expansion of genomic Schur functions}
\label{sec: main result}

In this section, we prove \cref{conj: Kim and Yoo}.
Throughout this section, we assume that $\lambda = (\lambda_1, \lambda_2)$ is a partition of $n$ and that $m$ is an integer with $l_\lambda \le m \le n$.

We begin by recalling Pechenik’s map \cite{14Pechenik, 20Pechenik}
from $\IGLTm{\lambda}{m}$ to $\bigcup_{\mu \in \Par(\lambda;m)} \SYT(\mu)$.

\begin{algorithm}\label{alg: Phi map}
Let $T \in \IGLTm{\lambda}{m}$.
\begin{enumerate}[label = {\it Step \arabic*.}, leftmargin = 10ex]
\item 
Let $A := \calI(T)$
and let $B$ be the set of entries that appear in the second row immediately to the right of an element of $A$.
\item 
Delete the elements of $A$ from the first row of $T$ and the elements of $B$ from the second row of $T$.
\item
Insert the elements of $B$ into the first column of $T$, preserving increasingness.
\item 
Denote the resulting tableau by $S_T$ and return it.
\end{enumerate}
\end{algorithm}
In \cite[Proposition~2.1]{14Pechenik} and \cite[Proposition~4.3]{20Pechenik}, it was shown that the map
\begin{align}\label{eq: Phi map def}
\Phi_{\lambda; m}: \IGLTm{\lambda}{m} \ra \bigcup_{\mu \in \Par(\lambda;m)} \SYT(\mu), 
\quad T \mapsto S_T
\end{align}
is a bijection.

An inverse algorithm for \cref{alg: Phi map} appears in the proof of \cite[Proposition~2.1]{14Pechenik} in the case $\lambda_1=\lambda_2$.
It extends naturally to the case $\lambda_1 \ge \lambda_2$ as follows.

\begin{algorithm}\label{alg: Inverse of Phi}
Let $S \in \bigcup_{\mu \in \Par(\lambda; m)} \SYT(\mu)$.
\begin{enumerate}[label = {\it Step \arabic*.}, leftmargin = 10ex]
\item 
Let $T$ be the tableau consisting of the first two rows of $S$, and let $B$ be the set of entries of $S$ lying below the first two rows.
\item 
Insert the elements of $B$ into the second row of $T$, preserving increasingness.
\item 
Let $A$ be the set of entries that appear immediately to the left of an element of $B$ in the second row of $T$.
If $|A| < \lambda_1+\lambda_2-m$, then include the greatest entry in the second row of $T$ to $A$. 
(Note that this element is not in $A$, since each element of $A$ lies to the left of some element of $B$ in the second row.)

\item 
Insert the elements of $A$ into the first row of $T$, again preserving increasingness.
Denote the resulting tableau by $T_S$ and return it.
\end{enumerate}
\end{algorithm}
The above algorithm yields
\begin{align*}
\Phi_{\lambda; m}^{-1}(S) = T_S \quad \text{for $S \in \bigcup_{\mu \in \Par(\lambda;m)} \SYT(\mu)$.}
\end{align*}
It is remarked in \cite[Proposition~4.3]{20Pechenik} that $\Phi_{\lambda; m}$ and $\Phi_{\lambda; m}^{-1}$ preserve descent sets.
This can be verified by a straightforward argument.

\begin{example}
Let
\[
T =
\begin{array}{l}
\begin{ytableau}
1 & 2 & 4 & 5 \\
2 & 3 & 5 & 6
\end{ytableau}
\end{array}
\in \IGLTm{(4,4)}{6}.
\]
We apply \cref{alg: Phi map} to $T$.
{\it Step~1} gives
\[
A = \{2,5\}
\quad \text{and} \quad
B = \{3,6\}.
\]
Continuing the algorithm, we have
\[
S_T =
\begin{array}{l}
\begin{ytableau}
1 & 4 \\
2 & 5 \\
3 \\
6
\end{ytableau}
\end{array}.
\]
One sees that 
\[
\Des(T) = \{1,2,4,5\} = \Des(S_T).
\]

Conversely, let
\[
S =
\begin{array}{l}
\begin{ytableau}
1 & 4 \\
2 & 5 \\
3 \\
6
\end{ytableau}
\end{array}.
\]
We apply \cref{alg: Inverse of Phi} to $S$.
{\it Step 1} gives
\[
T =
\begin{array}{l}
\begin{ytableau}
1 & 4 \\
2 & 5
\end{ytableau}
\end{array}
\quad \text{and} \quad
B = \{3,6\}.
\]
In Step~2, inserting the elements of $B$ into $T$ gives
\[
T =
\begin{array}{l}
\begin{ytableau}
1 & 4 \\
2 & 3 & 5 & 6
\end{ytableau}
\end{array}.
\]
In {\it Step 3}, we have $A = \{2,5\}$.
{\it Step 4} then produces
\[
T_S =
\begin{array}{l}
\begin{ytableau}
1 & 2 & 4 & 5 \\
2 & 3 & 5 & 6
\end{ytableau}
\end{array}.
\]
\end{example}

\begin{lemma}\label{lem: equiv class has same shape}
Let $E \in \calE_{\lambda; m}$.
For any $T_1, T_2 \in E$, $\sh(\Phi_{\lambda; m}(T_1)) = \sh(\Phi_{\lambda; m}(T_2))$.
\end{lemma}

\begin{proof}
Let $T_1, T_2 \in E$, and let $(A_1,B_1)$ and $(A_2,B_2)$ be the pairs of sets obtained by applying \cref{alg: Phi map} to $T_1$ and $T_2$, respectively.
Since $T_1, T_2 \in E$, we have 
\[
|A_1| = |\calI(T_1)| = |\calI(T_2)| = |A_2|.
\]
Moreover, each $i_1 \in \calI(T_1)$ corresponds to a unique $i_2 \in \calI(T_2)$ with $T_1^{-1}(i_1)=T_2^{-1}(i_2)$.
Thus $|B_1| = |B_2|$.
For $i=1,2$, the shape $\sh(\Phi_{\lambda; m}(T_i))$ depends only on $|A_i|$ and $|B_i|$.
Therefore, $\sh(\Phi_{\lambda; m}(T_1)) = \sh(\Phi_{\lambda; m}(T_2))$.
\end{proof}

We recall from \cref{eq: lambda(i)} the definitions of $\lambda^{(1)}_m$ and $\lambda^{(2)}_m$.
For $x = 1,2$, define
\[
\calE_{\lambda;m}^{(x)} := \{E \in \calE_{\lambda;m} \mid \sh(\Phi_{\lambda;m}(T)) = \lambda_m^{(x)} \text{ for all $T \in E$} \},
\]
By \cref{lem: equiv class has same shape}, these sets are well defined.
The following lemma characterizes when an equivalence class $E$ lies in $\calE_{\lambda;m}^{(1)}$ or $\calE_{\lambda;m}^{(2)}$.

\begin{lemma}
Let $E \in \calE_{\lambda;m}$.
Then
\[
E \in \calE_{\lambda;m}^{(1)}
\quad \text{if and only if} \quad
\text{$T((2, \lambda_2)) \notin \calI(T)$ for all $T \in E$.}
\]
Equivalently,
\[
E \in \calE_{\lambda;m}^{(2)}
\quad \text{if and only if} \quad \text{$T((2, \lambda_2)) \in \calI(T)$ for all $T \in E$.}
\]
\end{lemma}

\begin{proof}
By \cref{lem: equiv class has same shape}, it suffices to show that
\[
E\in \calE_{\lambda;m}^{(1)}
\quad \text{if and only if} \quad
T((2,\lambda_2))\notin \calI(T)\ \text{for some } T\in E.
\]

Let $T \in E$, and let $(A,B)$ be the pair of sets obtained by applying \cref{alg: Phi map} to $T$.
By the construction of $S_T$, we have $|A|=k_m$ and
\[
\sh(\Phi_{\lambda;m}(T))=(\lambda_1-|A|,\lambda_2-|B|,1^{|B|}).
\]
Moreover, by the choice of $B$, we have $|A|=|B|$ if and only if the rightmost box in the second row, namely $(2,\lambda_2)$, does not contain an element of $A$.
Thus,
\[
\sh(\Phi_{\lambda;m}(T))=(\lambda_1-k_m,\lambda_2-k_m,1^{k_m})
\quad \text{if and only if} \quad
T((2,\lambda_2))\notin \calI(T).
\]
This proves the assertion.
\end{proof}

We now collect the notation and definitions needed to prove \cref{conj: Kim and Yoo}.
For $x=1,2$, we define a partial order $\preceq^{(x)}$ on $\calE_{\lambda;m}^{(x)}$ as follows.
Let $E_1,E_2 \in \calE_{\lambda;m}^{(x)}$, and fix representatives $T_1 \in E_1$ and $T_2 \in E_2$.
Write
\[
\calI(T_1)=\{i_1<i_2<\cdots<i_{n-m}\}
\quad\text{and}\quad
\calI(T_2)=\{j_1<j_2<\cdots<j_{n-m}\}.
\]
We define
\begin{align}\label{eq: def of order(x)}
E_1 \preceq^{(x)} E_2
\quad \text{if} \quad
c_\sfb^{(i_k)}(T_1)\le c_\sfb^{(j_k)}(T_2)
\quad\text{for all }1\le k\le n-m.
\end{align}
Fix a linear extension $\preceq^{(x)}_e$ of $\preceq^{(x)}$ and enumerate the elements of $\calE_{\lambda;m}^{(x)}$ by
\begin{align}\label{eq: enumeration of E}
E^{(x)}_1 \preceq^{(x)}_e E^{(x)}_2 \preceq^{(x)}_e \cdots \preceq^{(x)}_e E^{(x)}_{|\calE_{\lambda;m}^{(x)}|}.
\end{align}

We extend the map $\Phi_{\lambda;m}$ in \cref{eq: Phi map def} to a linear map
\[
\Psi_{\lambda;m}:\bigoplus_{E\in\calE_{\lambda;m}}\bfG_E \ra X_{\lambda^{(1)}}\oplus X_{\lambda^{(2)}},
\quad T \mapsto \Phi_{\lambda;m}(T).
\]
Since $\Phi_{\lambda; m}: \IGLTm{\lambda}{m} \ra \bigcup_{\mu \in \Par(\lambda;m)} \SYT(\mu)$ is a bijection, so is $\Psi_{\lambda; m}$.

For $x=1,2$, set $M^{(x)}_0:=\{0\}$.
For $1\le j\le |\calE_{\lambda;m}^{(x)}|$, define a $\C$-vector space
\begin{align}\label{def of M_j}
M^{(x)}_j:=\Psi_{\lambda;m}\!\left(\bigoplus_{l=1}^j \bfG_{E^{(x)}_l}\right).
\end{align}
We are now ready to prove \cref{conj: Kim and Yoo}.

\begin{theorem}\label{thm: main}
Let $x \in \{1,2\}$ and $1 \le j \le |\calE_{\lambda;m}^{(x)}|$.
\begin{enumerate}[label = {\rm (\arabic*)}]
\item
The vector space $M^{(x)}_{j}$ is an $H_m(0)$-submodule of $X_{\lambda^{(x)}}$. 
\item 
The map
\[
\oPsi_{\lambda; m}: \bfG_{E_j^{(x)}} \ra M^{(x)}_{j} / M^{(x)}_{j-1}, 
\quad T \mapsto \Psi_{\lambda; m}(T) + M^{(x)}_{j-1}
\]
is an $H_m(0)$-module isomorphism.
\end{enumerate}
\end{theorem}

\begin{proof}
(1) 
By the construction of $M_j^{(x)}$, it is clear that $\{\Phi_{\lambda;m}(T) \mid T \in \bigcup_{l = 1}^j E_l^{(x)}\}$ is a basis of $M_j^{(x)}$.
Thus it suffices to show that
\[
\pi_i \cdot \Phi_{\lambda;m}(T) \in M_j^{(x)}
\quad \text{for all $1 \le i \le m-1$ and $T \in \bigcup_{l = 1}^j E_l^{(x)}$.}
\]

Fix $1 \le i \le m-1$ and $T \in \bigcup_{l = 1}^j E_l^{(x)}$.
For $1 \le k \le m$, let 
\begin{align*}
\row^{(k)}(T) := \{x \mid (x,y) \in T^{-1}(k)\}.
\end{align*}
We consider three cases according to the value of $\pi_i \cdot T$.
\medskip

\noindent
{\bf Case 1: $\pi_i \cdot T = T$.}
In this case, $i\notin \Des(T)$, and thus one of the following holds:
\begin{enumerate}[label = -, itemsep = 0.5ex]
\item 
$\row^{(i)}(T) = \row^{(i+1)}(T) = \{1\}$
\item 
$\row^{(i)}(T) = \row^{(i+1)}(T) = \{2\}$
\item 
$\row^{(i)}(T) = \{2\}$ and $\row^{(i+1)}(T) = \{1\}$
\item 
$\row^{(i)}(T) = \{2\}$ and $\row^{(i+1)}(T) = \{1,2\}$
\item 
$\row^{(i)}(T) = \{1, 2\}$ and $\row^{(i+1)}(T) = \{1\}$
\end{enumerate}
In each case, after applying \cref{alg: Phi map} to $T$, we see that 
$i$ lies strictly to the left of $i+1$ in $\Phi_{\lambda;m}(T)$.
Therefore, 
\[
\pi_i \cdot \Phi_{\lambda; m}(T)
= \Phi_{\lambda; m}(T) \in M_j^{(x)}.
\]

\medskip

\noindent
{\bf Case 2: $\pi_i \cdot T = 0$.}
In this case, one of the following holds:
\begin{enumerate}[label = {\rm (\roman*)}, itemsep = 0.5ex]
\item 
$\row^{(i)}(T) = \{1\}$, $\row^{(i+1)}(T) = \{2\}$, and $c^{(i)}(T) = c^{(i+1)}(T)$
\item
$\row^{(i)}(T) = \{1\}$ and $\row^{(i+1)}(T) = \{1,2\}$
\item
$\row^{(i)}(T) = \{1,2\}$ and $\row^{(i+1)}(T) = \{2\}$
or 
$\{1,2\}$ 
\end{enumerate}

In case (i),
$\row^{(i)}(T) = {1}$, $\row^{(i+1)}(T) = {2}$, and $c^{(i)}(T) = c^{(i+1)}(T)$.
Then the entry immediately to the left of $i+1$ in the second row of $T$, say $u$, occurs exactly once in $T$.
It follows that 
\[
r^{(i)}(\Phi_{\lambda; m}(T)) = 1 
\quad \text{and} \quad
r^{(i+1)}(\Phi_{\lambda; m}(T)) = 2.
\]
Let $A$ and $B$ be the sets defined in \cref{alg: Phi map}.
Since $|T^{-1}(u)| = 1$, by the definitions of $A$ and $B$, we have
\[
|\{ a \in A \mid c_{\sft}^{(a)}(T) < c_{\sft}^{(i)}(T) \}| 
= |\{ b \in B \mid c_{\sfb}^{(b)}(T) < c_{\sfb}^{(i+1)}(T) \}|.
\]
It follows that
\[
c^{(i)}(\Phi_{\lambda; m}(T)) = c^{(i+1)}(\Phi_{\lambda; m}(T)).
\]
Thus, $\pi_i \cdot \Phi_{\lambda; m}(T) = 0 = \Phi_{\lambda; m}(\pi_i \cdot T)$.
\medskip

In case (ii), $\row^{(i)}(T) = \{1\}$ and $\row^{(i+1)}(T) = \{1,2\}$.
Applying \cref{alg: Phi map} to $T$, we have
\[
c^{(i)}(\Phi_{\lambda; m}(T)) = c^{(i+1)}(\Phi_{\lambda; m}(T)) \quad \text{or} \quad 
c^{(i)}(\Phi_{\lambda; m}(T)) > c^{(i+1)}(\Phi_{\lambda; m}(T)).
\]
If $c^{(i)}(\Phi_{\lambda; m}(T)) = c^{(i+1)}(\Phi_{\lambda; m}(T))$, then \[
\pi_i \cdot \Phi_{\lambda; m}(T) = 0 \in M^{(x)}_j.
\]
Assume that $c^{(i)}(\Phi_{\lambda; m}(T)) > c^{(i+1)}(\Phi_{\lambda; m}(T))$.
Then $\pi_i \cdot \Phi_{\lambda; m}(T) = s_i \cdot \Phi_{\lambda; m}(T)$, and hence
\[
\sh(\pi_i \cdot \Phi_{\lambda; m}(T)) 
= \sh(s_i \cdot \Phi_{\lambda; m}(T))
= \sh(\Phi_{\lambda; m}(T)) = \lambda^{(x)}.
\]
Applying \cref{alg: Inverse of Phi} to $s_i \cdot \Phi_{\lambda; m}(T)$ yields, for $B \in \tyd(\lambda)$,
\[
\Phi_{\lambda; m}^{-1} (s_i \cdot \Phi_{\lambda; m}(T))(B) = \begin{cases}
i & \text{if $B = (r_{\sfb}^{(i+1)}(T),\, c_{\sfb}^{(i+1)}(T))$,} \\
T(B) & \text{otherwise.}
\end{cases}
\]
It follows that
\begin{align*}
\calI(\Phi_{\lambda; m}^{-1} (s_i \cdot \Phi_{\lambda; m}(T))) 
& = \calI(T) \cup \{i\} \setminus \{i+1\}, \\
c_\sft^{(i)}\left(
\Phi_{\lambda; m}^{-1} (s_i \cdot \Phi_{\lambda; m}(T)) 
\right)
& = c_\sft^{(i+1)}(T) - 1, \quad \text{and} \\
c_\sfb^{(i)}\left(
\Phi_{\lambda; m}^{-1} (s_i \cdot \Phi_{\lambda; m}(T))
\right)
& = c_\sfb^{(i+1)}(T).
\end{align*}
By the definition of $\preceq^{(x)}$ and the enumeration in \cref{eq: enumeration of E}, we have $\Phi_{\lambda; m}^{-1} (\pi_i \cdot \Phi_{\lambda; m}(T)) \in E_{l}^{(x)}$ for some $1 \le l < j$, and thus
\[
\pi_i \cdot \Phi_{\lambda; m}(T) \in M^{(x)}_l \subseteq M^{(x)}_j.
\]
\medskip

In case (iii), $\row^{(i)}(T) = \{1,2\}$ and $\row^{(i+1)}(T) = \{2\}$ or $\{1,2\}$.
Let $a$ be the entry immediately to the left of $i$ in the second row of $T$.
If $a \in \calI(T) \cup \{-\infty\}$, then
\[
c^{(i)}(\Phi_{\lambda; m}(T)) = c^{(i+1)}(\Phi_{\lambda; m}(T)) = 1,
\]
thus $\pi_i \cdot \Phi_{\lambda; m}(T) = 0 \in M^{(x)}_j$.
Suppose that $a \notin \calI(T) \cup \{-\infty\}$.
Then
\[
r^{(i)}(\Phi_{\lambda; m}(T)) = 2, \quad c^{(i)}(\Phi_{\lambda; m}(T)) > 1,
\quad \text{and} \quad
c^{(i+1)}(\Phi_{\lambda; m}(T)) = 1.
\]
It follows that $\pi_i \cdot \Phi_{\lambda; m}(T) = s_i \cdot \Phi_{\lambda; m}(T)$, so 
\[
\sh(\pi_i \cdot \Phi_{\lambda; m}(T)) = \sh(s_i \cdot \Phi_{\lambda; m}(T)) = \sh(\Phi_{\lambda; m}(T)) = \lambda^{(x)}.
\]
Let $w_0, w_1, \ldots, w_t$ be the entries in the first row of $T$ such that
\[
w_0 < a < w_1, \quad
c^{(w_{u+1})}_\sft(T) = c^{(w_u)}_\sft(T) + 1, \quad \text{and} \quad
c^{(w_{t})}_\sft(T) = c^{(i)}_\sft(T) - 1.
\]
Since $a$ is immediately to the left of $i$ in the second row of $T$, $w_1, \ldots, w_t$ cannot appear in the second row of $T$.
Thus, $|T^{-1}(w_u)| = 1$ for all $1 \le u \le t$.
Applying \cref{alg: Inverse of Phi} to $s_i \cdot \Phi_{\lambda; m}(T)$ yields, for $B \in \tyd(\lambda)$,
\[
\Phi_{\lambda; m}^{-1} (s_i \cdot \Phi_{\lambda; m}(T))(B) = \begin{cases}
a & \text{if $B = T^{-1}(w_1)$,} \\
w_p & \text{if $B = T^{-1}(w_{p+1})$ for $1 \le p \le t-1$,} \\
w_t & \text{if $B = (r_{\sft}^{(i)}(T),\, c_{\sft}^{(i)}(T))$,} \\
T(B) & \text{otherwise.}
\end{cases}
\]
It follows that
\begin{align*}
\calI(\Phi_{\lambda; m}^{-1} (s_i \cdot \Phi_{\lambda; m}(T))) 
& = \calI(T) \cup \{a\} \setminus \{i\}, \\
c_\sft^{(a)}\left(
\Phi_{\lambda; m}^{-1} (s_i \cdot \Phi_{\lambda; m}(T)) 
\right)
& < c_\sft^{(i)}(T), \quad \text{and} \\
c_\sfb^{(a)}\left(
\Phi_{\lambda; m}^{-1} (s_i \cdot \Phi_{\lambda; m}(T))
\right)
& = c_\sfb^{(i)}(T) - 1.
\end{align*}
By the definition of $\preceq^{(x)}$ and the enumeration in \cref{eq: enumeration of E}, we have $\Phi_{\lambda; m}^{-1} (\pi_i \cdot \Phi_{\lambda; m}(T)) \in E_{l}^{(x)}$ for some $1 \le l < j$, and thus
\[
\pi_i \cdot \Phi_{\lambda; m}(T) \in M^{(x)}_l \subseteq M^{(x)}_j.
\]
\smallskip

\noindent
{\bf Case 3: $\pi_i \cdot T = s_i \cdot T$.}
In this case, 
\[
\row^{(i)}(T) = \{1\}, 
\quad
\row^{(i+1)}(T) = \{2\}, 
\quad \text{and} \quad 
c^{(i)}(T) > c^{(i+1)}(T).
\]
Applying \cref{alg: Phi map} to $T$ yields
\[
r^{(i)}(\Phi_{\lambda; m}(T)) = 1, 
\quad
r^{(i+1)}(\Phi_{\lambda; m}(T)) > 1, 
\quad \text{and} \quad 
c^{(i)}(\Phi_{\lambda;m}(T)) > c^{(i+1)}(\Phi_{\lambda;m}(T)).
\]
Thus, $\pi_i \cdot \Phi_{\lambda; m}(T) = s_i \cdot \Phi_{\lambda; m}(T)$.
Since there is no integer between $i$ and $i+1$, applying \cref{alg: Phi map} to $s_i \cdot T$ yields 
\[
\Phi_{\lambda; m}(s_i \cdot T) = s_i \cdot \Phi_{\lambda; m}(T).
\]
Thus, $\pi_i \cdot \Phi_{\lambda; m}(T) = \Phi_{\lambda; m}(\pi_i \cdot T)$.
\bigskip

(2) We have shown in the proof of (1) that for $1 \le i \le m-1$ and $T \in  E_j^{(x)}$, the following statements hold:
\begin{enumerate}[label = {\rm (\roman*)}, itemsep = 0.5ex]
\item 
If $\pi_i \cdot T = T$, then $\pi_i \cdot \Psi_{\lambda; m}(T) = \Psi_{\lambda; m}(\pi_i \cdot T)$.
\item
If $\pi_i \cdot T = 0$, then either $\pi_i \cdot \Psi_{\lambda; m}(T) = 0$ or $\pi_i \cdot \Psi_{\lambda; m}(T) \in M^{(x)}_l$ for some $1 \le l < j$.
\item
If $\pi_i \cdot T = s_i \cdot T$, then $\pi_i \cdot \Psi_{\lambda; m}(T) 
= \Psi_{\lambda; m}(\pi_i \cdot T)$.
\end{enumerate}
Moreover, since $\Psi_{\lambda;m}$ is a bijection,
the map
\[
\oPsi_{\lambda; m}: \bfG_{E_j^{(x)}} \ra 
\left(\bigoplus_{l=1}^j \bfG_{E^{(x)}_l}\right) \bigg/ \left(\bigoplus_{l=1}^{j-1} \bfG_{E^{(x)}_l}\right),
\quad T \mapsto \Psi_{\lambda; m}(T) + \left(\bigoplus_{l=1}^{j-1} \bfG_{E^{(x)}_l}\right)
\]
is also a bijection.
Therefore, $\oPsi_{\lambda; m}$ is an $H_m(0)$-module isomorphism.
\end{proof}

We conclude by explaining why \cref{thm: main} implies \cref{conj: Kim and Yoo}. 
By definition, the sets $\calE_{\lambda;m}^{(1)}$ and $\calE_{\lambda;m}^{(2)}$ form a partition of $\calE_{\lambda;m}$. 
Fix $x \in \{1,2\}$ such that $\lambda_m^{(x)} \neq \emptyset$, and set $\mu := \lambda_m^{(x)}$ and $\calE_\mu := \calE_{\lambda;m}^{(x)}$. 
Then, for any total order chosen in \cref{eq: enumeration of E}, \cref{thm: main}(2) shows that the $H_m(0)$-modules defined in \cref{def of M_j} give a filtration
\[
M_{0}^{(x)} = \{0\} \subseteq M_{1}^{(x)} \subseteq \cdots \subseteq M_{|\calE_{\lambda;m}^{(x)}|}^{(x)} = X_\mu
\]
whose successive quotients are isomorphic to the corresponding modules $G_{E_j^{(x)}}$. 
Thus condition {\bf C2} of \cref{conj: Kim and Yoo} holds. 
Moreover, since $\ch([X_\mu]) = s_\mu$, the additivity of the quasisymmetric characteristic along the above filtration yields
$\sum_{E \in \calE_\mu} \ch([G_E]) = s_\mu$.
Hence condition {\bf C1} also holds, and therefore \cref{conj: Kim and Yoo} follows.


\bibliographystyle{abbrv}
\bibliography{references}

\end{document}